\documentclass[a4paper,12pt]{elsarticle}
\usepackage{amssymb,amsmath,amsthm}

\makeatletter
\def\ps@pprintTitle{%
  \let\@oddhead\@empty
  \let\@evenhead\@empty
  \def\@oddfoot{\reset@font\hfil\thepage\hfil}
  \let\@evenfoot\@oddfoot
}
\makeatother

\newcommand\cH{{\mathcal H}}
\newcommand\cF{{\mathcal F}}
\newcommand\cA{{\mathcal A}}
\newcommand\cB{{\mathcal B}}
\newcommand\cY{{\mathcal Y}}
\newcommand\cG{{\mathcal G}}
\newcommand\cD{{\mathcal D}}
\newcommand\eps{{\varepsilon}}
\newcommand\cU{{\mathcal U}}
\newcommand\junk[1]{}

\theoremstyle{plain}
\newtheorem{theorem}{Theorem}
\newtheorem{corollary}[theorem]{Corollary}
\theoremstyle{definition}
\newtheorem{defn}[theorem]{Definition}
\newtheorem{con}[theorem]{Construction}
\newtheorem{fact}[theorem]{Fact}
\newtheorem{remark}[theorem]{Remark}

\newcommand\tref[1]{Theorem~\ref{thm:#1}}
\newcommand\sref[1]{Section~\ref{sec:#1}}
\newcommand\fref[1]{Fact~\ref{fact:#1}}

\begin{document}
\begin{frontmatter}
\title{Two-part set systems}
\author[renyi]{P\'eter L. Erd\H os\fnref{elp}}
\author[renyi]{D\'aniel Gerbner\fnref{dani}}
\author[dhruv]{ Dhruv Mubayi\fnref{dhruv1}}
\author[renyi]{Nathan Lemons\fnref{dani}}
\author[renyi,pal]{Cory Palmer\fnref{dani}}
\author[renyi]{Bal\'azs Patk\'os\fnref{pat}}
\address[renyi]{Alfr\'ed R{\'e}nyi Institute, Re\'altanoda u 13-15 Budapest, 1053 Hungary\\
        {\tt email}: $<$elp,gerbner,nathan,palmer,patkos$>$@renyi.hu}
\address[dhruv]{Univ. Illinois at Chicago, Chicago, IL 60607, US; {\tt email}:
        mubayi@math.uic.edu}
\address[pal]{Univ. Illinois, Urbana-Champaign, Urbana, IL 61801, US}
\fntext[elp]{Research supported in part by the Hungarian NSF, under contract NK 78439 and K 68262}
\fntext[dani]{Research supported in part by the Hungarian NSF, under contract NK 78439}
\fntext[dhruv]{Research supported in part by NSF grant DMS-0969092}
\fntext[pat]{Research supported by Hungarian NSF, under contract PD-83586, and the J\'anos Bolyai Research Scholarship of the Hungarian Academy of Sciences.}

\begin{abstract}
The two part Sperner theorem of Katona and Kleitman states that if $X$ is an $n$-element set with partition $X_1 \cup X_2$, and $\cF$ is a family of subsets of $X$ such that  no two sets $A, B \in \cF$  satisfy $A \subset B$ (or $B \subset A$) and $A \cap X_i=B \cap X_i$ for some $i$, then $|\cF| \le {n \choose \lfloor n/2 \rfloor}$. We consider variations of this problem by replacing the Sperner property with the intersection property and considering families that satisfiy various combinations of these properties on one or both parts $X_1$, $X_2$. Along the way, we prove the following  new result which may be of independent interest: let $\cF, \cG$ be families of subsets of an $n$-element set such that $\cF$ and $\cG$ are both intersecting and cross-Sperner, meaning that if $A \in \cF$ and $B \in \cG$, then  $A \not\subset B$ and $B \not\subset A$.
Then  $|\cF| +|\cG| < 2^{n-1}$ and there are exponentially many examples showing that this bound is tight.\end{abstract}

\end{frontmatter}

\section{Introduction}
Let $X$ be a finite set and let $2^X$ be the system of all subsets
of $X$. The basic problem of the theory of extremal sets systems
is to determine the maximum size that a set system $\cF \subseteq
2^{X}$ can have provided $\cF$ satisfies a prescribed property.
The prototypes of investigated properties are the intersecting and
Sperner properties. A set system $\cF$ is \textit{intersecting} if
$F_1 \cap F_2 \neq \emptyset$ for any pair $F_1,F_2 \in \cF$ and a
set system $\cF$ is \textit{Sperner} if there do not exist two
distinct sets $F_1,F_2 \in \cF$ such that $F_1 \subset F_2$. The
celebrated theorems of Erd\H os, Ko, Rado \cite{EKR} and of
Sperner \cite{S} determine the largest size that a uniform
intersecting set system and Sperner system can have. Both theorems
have many applications and generalizations.

One such generalization of the Sperner property is the so called
\textit{more part Sperner property}. In this case, the underlying
set $X$ is partitioned into $m$ subsets $X_1,...,X_m$ and the
system $\cF \subset 2^X$ is said to be $m$-part Sperner if for any
pair $F_1,F_2 \in \cF$ with $F_1 \subset F_2$ there exist at least
two indices $1 \le i_1< i_2 \le m$ such that $F_1 \cap X_{i_j}
\subsetneq F_2 \cap X_{i_j}$ holds for $j=1,2$. Systems with this
property were first considered in \cite{Ka,K2}; for a survey of recent results
see \cite{ACES}.

In this paper we will consider analogous problems for intersection
properties and also some mixed more part properties in the case
when $m$ equals 2. All maximum size 2-part Sperner set systems
were described by P.L. Erd\H{o}s and G.O.H. Katona in
\cite{EK1,EK2}. To rephrase the 2-part Sperner property it is
convenient to introduce the following set systems of
\textit{traces}: for any $A \subseteq X_1$ and $B \subseteq X_2$
let $\cF_A=\{F \cap X_2: F \in \cF, F \cap X_1=A\}, \cF_B=\{F \cap
X_1: F \in \cF, F \cap X_2=B\}$. Also, for any $F \in \cF$ we will
call $F \cap X_1$ and $F \cap X_2$ the \emph{traces} of $F$ on
$X_1$ and $X_2$. One can easily see that a set system $\cF$ is
2-part Sperner with respect to the partition $X=X_1 \cup X_2$ if
and only if for any subset $A \subseteq X_1$ or $B \subseteq X_2$
the set systems $\cF_A$ and $\cF_B$ possess the Sperner property.

Having this equivalence in mind, it is natural to introduce the
following three definitions where we always assume that the
underlying set $X$ is partitioned into two sets $X_1$ and $X_2$:

\begin{defn}
(i) a set system $\cF \subseteq 2^X$ is 2-part intersecting (a
2I-system for short) if for any subset $A$ of $X_1$ (and for any
subset $B$ of $X_2$) the trace system $\cF_A$ on $X_2$ (and the
trace system $\cF_B$ on $X_1$) is intersecting,

(ii) a set system $\cF \subseteq 2^X$ is 2-part intersecting,
2-part Sperner (a 2I2S-system for short) if for any subset $A$ of
$X_1$ and for any subset $B$ of
$X_2$ the trace systems $\cF_A$ on $X_2$ and $\cF_B$ on $X_1$ are intersecting and
Sperner,

(iii) a set system $\cF \subseteq 2^X$ is 1-part intersecting,
1-part Sperner (a 1I1S-system for short) if there exists no pair
of distinct sets $F_1,F_2$ in $\cF$ such that the traces of
$F_1,F_2$ are disjoint at one of the parts and are in containment at
the other.
\end{defn}

We will address the problem of finding the maximum possible size
of a set system possessing the properties above. Some of our
bounds will apply regardless of the sizes of the parts in the
2-partition and some will only apply to special cases. We will be
mostly interested in the case when $|X_1|=|X_2|$. Clearly, for any
2-part set system $\cF$ we have $|\cF|=\sum_{A \subseteq
X_1}|\cF_A|=\sum_{B \subseteq X_2}|\cF_B|$. As any intersecting
system of subsets of $X_1$ has size at most $2^{|X_1|-1}$, it
follows that any 2I-system has size at most
$2^{|X_2|}2^{|X_1|-1}=2^{|X|-1}$. In \sref{2Isec} we will prove
the following theorem.

\begin{theorem}
\label{thm:2I} Let $\cF \subseteq 2^X$ be a 2-part intersecting
system of maximum size. If the 2-partition $X=X_1 \cup X_2$ is
non-trivial (i.e. $X_1 \neq \emptyset$, $X_2 \neq \emptyset$),
then the following inequality holds:
$$|\cF| \leq \frac{3}{8}2^{|X|}.$$
The bound is best possible if $X_1$ or $X_2$ is a singleton.
Moreover, if $|X_1|=|X_2|$, then there exists a 2-part
intersecting system of size $\frac{1}{3} \left( 2^{|X|} +2 \right)$.
\end{theorem}

The rest of \sref{2Isec} is devoted to 2I2S systems. We prove the following result.

\begin{theorem}
\label{thm:2I2S} Let $\cF \subseteq 2^X$ be a 2-part intersecting,
2-part Sperner system of maximum size. Then $|\cF| \le
\binom{|X|}{\lceil |X|/2\rceil}$ holds. This bound is
asymptotically sharp as long as $|X_1|=o(|X_2|^{1/2})$. If
$|X_1|=|X_2|$ holds, then there exists a 2I2S system of size
$c\binom{|X|}{\lceil |X|/2\rceil}$ with $c>2/3$.
\end{theorem}

The main result of the paper is proved in \sref{1I1Ssec}. We
determine the maximum size of a 1-part intersecting 1-part
Sperner set system.

\begin{theorem}\label{thm:1I1S}
Let $\cF$ be a maximum size 1-part intersecting, 1-part Sperner
set system. Then $|\cF|=2^{|X|-2}.$
\end{theorem}

\section{2I- and 2I2S-systems}
\label{sec:2Isec}

In this section we consider two-part intersecting and two-part
intersecting, two-part Sperner set systems. We first consider a
general construction that produces large families with these properties. Let
$\cA_1,...,\cA_m$ and $\cB_1,...,\cB_m$ be partitions of $2^{X_1}$
and $2^{X_2}$  into disjoint intersecting (or intersecting, Sperner)
systems some of which may possibly be empty. Then the set system
$\cF:=\cup_{i=1}^m \cA_i \times \cB_i=\{A \cup B: A \in \cA_i,
B\in \cB_i \hskip 0.2truecm \text{for some}\ 1 \le i \le m\}$ is a
2I- (2I2S)-system by definition.

\vskip 0.3truecm

\begin{fact}\label{fact}
Let $0 \le x_1 \le ... \le x_n$, $0 \le y_1 \le ... \le y_n$ be
real numbers and $\pi$ be a permutation of the first $n$ integers.
Then we have the following inequalities:
$$
\sum_{i=1}^nx_iy_{\pi(i)} \le \sum_{i=1}^n x_i y_i \le
\max\left\{\sum_{i=1}^nx^2_i, \sum_{i=1}^ny_i^2\right\}.
$$
\end{fact}

Thus to maximize the size of a family obtained through the general construction one
should enumerate the $\cA_i$'s and the $\cB_i$'s in decreasing
order according to their size. Moreover, if $|X_1|=|X_2|$, then it
is enough to consider partitions $\cA_1,...,\cA_m$ of $2^{X_1}$
and the sum $\sum_{i=1}^m|\cA_i|^2$.

\subsection{Two-part intersecting systems}

In this subsection we prove \tref{2I}. In the proof we
use the following theorem of Kleitman \cite{K}.

\begin{theorem}[Kleitman \cite{K}]
\label{thm:kleit}
Let $\cF_1,...,\cF_m \subseteq 2^{[n]}$ be intersecting set systems. Then
$$
|\cF_1 \cup ... \cup \cF_m| \leq 2^n-2^{n-m}.
$$
\end{theorem}

\noindent {\em Proof of \tref{2I}.} For any subset $A$ of $X_1$
let $\bar{A}$ denote its complement $X_1 \setminus A$. By
definition, both $\cF_A$ and $\cF_{\bar{A}}$ are intersecting.
Also, these set systems are disjoint as $B \in \cF_A \cap
\cF_{\bar{A}}$ implies $B \cup A, B \cup \bar{A} \in \cF$
which contradicts the 2-part intersecting property of $\cF$.
Thus by \tref{kleit} we have $|\cF_{A}|+|\cF_{\bar{A}}| \leq
2^{|X_2|-1}+2^{|X_2|-2}$.

Altogether we obtain
\begin{displaymath}
|\cF| \leq 2^{|X_1|-1} ( 2^{|X_2|-1} + 2^{|X_2|-2} ) = \frac{3}{8} 2^{|X|}.
\end{displaymath}

Our best lower bounds arise from our general construction. If $X_1$ consists of a single element
$x_1$, then let $\cA_1=\{\{x_1\}\}, \cA_2=\{\emptyset\}$ and
$\cB_1=\{B \subset X_2: x_2 \in B\}, \cB_2=\{B \subset X_2: x_2
\notin B, x_2' \in B\}$ for two fixed elements $x_2,x_2' \in X_2$
and the other $\cB_i$'s be arbitrary while the other $\cA_i$'s be
empty. For the set system $\cF$ we obtain via the general
construction, we have
$|\cF|=2^{|X_2|-1}+2^{|X_2|-2}=\frac{3}{8}2^{|X|}$.

Finally, let us suppose that $|X_1|=|X_2|=|X|/2$ and let the
elements of $X_1$ and $X_2$ be $x^1_1,...,x^1_m$ and
$x^2_1,...,x^2_m$. Let us define the partition of $2^{X_1}$ and
$2^{X_2}$ in the following way: $\cA_i:=\{A \subset X_1\setminus
\{x^1_1,...,x^1_{i-1}\}: x^1_i \in A\}, \cB_i:=\{B \subset
X_2\setminus \{x^2_1,...,x^2_{i-1}\}: x^2_i \in B\}$ for all $1
\le i \le m+1$ (i.e. $\cA_{m+1}=\cB_{m+1}=\{\emptyset \}$). Then
for the set system $\cF$ arising from the general construction we
have
$$
|\cF|=1+\sum_{i=1}^{m}2^{|X|-2i}=\frac{2^{|X|}+2}{3}. \hspace{2cm}
\Box
$$

\begin{remark}\label{rem7}
\tref{kleit} shows that the above set system for the $|X_1|=|X_2|$
case is best possible among those that we can obtain via the
general construction. Indeed, by  Fact~\ref{fact}  we know that we
have to consider partitions of $2^{X_1}$ to intersecting set
systems with sizes $s_1,s_2, ...,s_m$ and maximize
$\sum_{i=1}^ms_i^2$. But a partition maximizes this sum of squares
if for all $1 \le j \le m$ the sums $\sum_{i=1}^js_i$ are
maximized. In the construction we use, the sums $\sum_{i=1}^js_i$
match the upper bound of \tref{kleit}.
\end{remark}

\subsection{Two-part intersecting, two-part Sperner systems}

In this subsection we consider 2I2S-systems and prove \tref{2I2S}.
To be able to use the general construction, we need to define a
partition of the power set into intersecting Sperner set systems.

\begin{con}
\label{con:partIS} Here we give a partition of the power set of
$Y$ into intersecting Sperner systems where all levels are partitioned
into minimal number of (uniform) intersecting systems (we call this
\textit{canonical partition}). This
partition is in the form of
\begin{eqnarray*}
\cY_k, &\hbox { for }& k=\left \lceil \frac{|Y|+1}{2} \right
\rceil,\ldots,|Y|;
\\
\cY_{i,j}, &\hbox { for }& i=1,\ldots, \left \lceil
\frac{|Y|+1}{2} \right \rceil -1, j=1,\ldots,|Y|-2i+1;
 \\
\cY^*_\ell, &\hbox { for }& \ell=0,\ldots,\left \lceil
\frac{|Y|+1}{2} \right \rceil -1.
\end{eqnarray*}
The systems $\cY_k$ are ${Y \choose k}$. Fix an enumeration
$y_1,...,y_{|Y|}$ of the elements of $Y$ and define the systems
$\cY_{i,j}$ as $\left \{Y' \in \binom{Y\setminus \{y_1,...,y_{j-1}
\}}{i}: y_j \in Y'\right \}.$ Finally let $\cY^*_\ell=\binom{Y}{\ell}
\setminus \bigcup_{j=1}^{|Y|-2\ell +1}\cY_{\ell ,j}$. We remark
that the second and third types are identical to those in the
corresponding Kneser construction. Note that the number of systems
in the partition is quadratic in $|Y|$ but for any $\eps >0$ there exists $K=K(\eps)$ such that
\begin{equation} \label{union}
\left|\bigcup_{k=\left \lceil \frac{|Y|+1}{2} \right\rceil}^{|Y|} \cY_k \cup
\bigcup_{i=|Y|/2-K|Y|^{1/2}}^{|Y|/2}\bigcup_{j=1}^{|Y|-2i+1}\cY_{i,j} \cup \bigcup_{\ell=|Y|/2-K|Y|^{1/2}}^{|Y|/2}\cY^*_{\ell} \right| \ge (1-\eps)2^{|Y|}.
\end{equation}
Indeed,  the sets in all the $\cY_k$ contain all subsets of $Y$ of size greater than $|Y|/2$, and the remaining families
$\cY_{i,j}, \cY_{\ell}^*$ contain all subsets of $Y$ of size between $|Y|/2 - K|Y|^{1/2}$ and $|Y|/2$.  Since the number of subsets of $Y$ of size less than $|Y|/2-K|Y|^{1/2}$ is less than $\eps 2^{|Y|}$, the inequality  in (\ref{union}) follows.  It is easy to see that the number of set systems in the  union in (\ref{union}) is at most $2K^2|Y|$.

\end{con}

\begin{proof}[Proof of \tref{2I2S}]
The upper bound of the theorem follows from the result of Katona
\cite{Ka} and Kleitman \cite{K2} stating that a 2-part Sperner
system has size at most $\binom{|X|}{\lceil |X|/2 \rceil}$, since
any 2I2S-system is 2-part Sperner.

\smallskip
We now prove the lower bound. For $i=1,2$ let $x_i=|X_i|$, and recall that $n=|X|=x_1+x_2$. First we consider the case when the size of $x_1$ is negligible
compared to the size of $x_2$. Let us assume that $x_1=
o(x_2^{1/2})$.  As observed above, from the canonical partition of $2^{X_1}$ which has $\Theta(x_1^2)$ families, there are $m=O(x_1)$ families $\cF^1_1,...,\cF^1_m\subset 2^{X_1}$ such that
$$\left|\bigcup_{i=1}^m \cF^1_i\right| = (1-o(1))2^{x_1}.$$If $i=o(x_2^{1/2})$, then the system $\binom{X_2}{x_2/2+i}$ is
intersecting Sperner and has size $ (1 - o(1))
\binom{x_2}{x_2/2} = 1/2^{x_1}(1-o(1))\binom{n}{n/2}$.
Thus, by the general constuction, we obtain the following 2I2S-system from these partitions:
\[
\cF=\bigcup_{i=1}^m\left\{F \cup H: F\in \cF^1_i, H \in \binom{X_2}{x_2/2+i}\right\}.
\]
By the above, $|\cF|$ is equal to
$$\sum_{i=1}^m |\cF^1_i|{x_2 \choose \frac{x_2}{2}+i}\ge \frac{1}{2^{x_1}}(1-o(1)){n \choose \frac{n}{2}}\sum_{i=1}^m|\cF_i^1|=(1-o(1)){n \choose \frac{n}{2}}.$$

Let us consider the case $x_1=x_2$. We first show that
the 2I2S-system $\cF$ we derive from the canonical partition using
our general construction has size $(2/3-o(1))\binom{n}{\lceil
n/2\rceil}$. We then use Frankl and F\"uredi's construction \cite{FF} to improve this bound by a constant factor. For sake of simplicity, assume $n$ is divisible
by 4. Then our system has size
$$
\sum_{i=n/4+1}^{n/2}\binom{n/2}{i}^2+\sum_{i=1}^{n/4}\sum_{k=0}^{n/2-2i}
\binom{n-1-k}{i-1}^2+\sum_{i=1}^{n/4}\binom{2i-1}{i}^2,
$$
where the sums belong to the three different system types in the
canonical partition. We can write our system as $\cF = \cF_1 \cup
\cF_2$ where the first subsystem corresponds to the sets listed in
the first summation, and the second one consists of the other
sets. Then
\begin{eqnarray*}
|\cF_1|=\sum_{i=n/4+1}^{n/2}\binom{n/2}{i}^2 &=& \sum_{i=n/4+1}^{n/2}
\binom{n/2}{i} \binom{n/2}{n/2-i} \\
&=& 1/2\binom{n}{n/2}-\binom{n/2}{n/4}^2= (1/2-o(1))\binom{n}{n/2}
\end{eqnarray*}
as $\binom{n/2}{i}\binom{n/2}{n/2-i}$ is the number of those
$n/2$-subsets of $X$ that intersect $X_1$ in $i$ elements.

\bigskip \noindent Next we prove that  $|\cF_2|\ge (1/3-o(1))|\cF_1|$  which implies that $|\cF_2|\ge \frac{ 1/2- o(1)}{3}\binom{n}{n/2}$ and thus $|\cF|\ge(\frac{2}{3}-o(1)\binom{n}{n/2}$.
We consider those members of $\cF_2$ which intersect $X_1$ in
$i$ elements (and then intersect $X_2$ in $i$ elements too). We
will show that, for most values of $i$, the number of these sets is roughly a third of
the number of those members of $\cF_1$, which intersect $X_1$ (and
then $X_ 2$ as well) in $n/2 - i$ elements. We have to compare

$$S_i=\binom{2i-1}{i}^2+\sum_{k=0}^{n/2-2i}\binom{n/2-1-k}{i-1}^2
\quad \hbox{to} \quad \binom{n/2}{n/2-i}^2=\binom{n/2}{i}^2.
$$
We will be done, if we establish $S_i/(\binom{n/2}{i})^2=1/3+o(1)$ for all $n/4-n^{2/3} \le i \le n/4-\log n$ as
\[
\sum_{i<n/4-n^{2/3}}\binom{n/2}{n/2-i}^2+\sum_{n/4-\log n<i\le n/2}\binom{n/2}{n/2-i}^2=o\left(\binom{n}{n/2}\right).
\]

\vskip 0.3truecm

To deduce $S_i/(\binom{n/2}{i})^2=1/3+o(1)$ we need the following fact.

\vskip 0.2truecm
\begin{fact}
\label{fact:sum}Let $a_1 \ge a_2 \ge ... \ge a_k >0$ positive reals with $\sum_{\ell=1}^ka_{\ell}=1$. If for some $j<k$ we have $a_{\ell}=2^{-\ell}+o(1)$ for all $\ell<j$ and $\sum_{\ell=j}^ka_{\ell}=o(1)$, then $\sum_{\ell=1}^ka_{\ell}^2=1/3+o(1)$.
\end{fact}

All we have to do is to verify the conditions of \fref{sum} to the numbers
\[
r_{\ell}=\frac{\binom{n/2-\ell}{i-1}}{\binom{n/2}{i}} \hskip 0.3truecm \text{for}\ \ell=1,...,n/2-2i+1 \hskip 0.3truecm \text{and}\ r_{n/2-2i+2}=\frac{\binom{2i-1}{i}}{\binom{n/2}{i}}
\]
with $j=\min\{n/4-i,n^{1/4}\}$ and $k=n/2-2i+2$. First of all $\sum_{\ell}r_{\ell}=1$ as these numbers correspond to the ratios of set systems in a partition. Next we show that $r_{\ell}=2^{-\ell}+o(1)$ for all $\ell<j$.
Writing $d_{\ell}=\frac{r_{\ell}}{r_{\ell-1}}$ for $2\le \ell \le j-1$ and $i=n/4-m$ we obtain
\[
d_{\ell}=\frac{r_{\ell}}{r_{\ell-1}}=\frac{\binom{n/2-\ell}{i-1}}{\binom{n/2-\ell+1}{i-1}}=\frac{n/2-\ell+i+2}{n/2-\ell+1}=\frac{1}{2}+\frac{m-\ell/2+3/2}{n/2-\ell+1}=\frac{1}{2}+O(n^{-1/3})
\]
and thus for $\ell <j\le n^{1/4}$
\[
r_1=\frac{i}{n/2}=\frac{1}{2}+o(1) \hskip 0.2truecm \text{and}\ r_{\ell}=r_1\prod_{t=2}^{\ell}d_{t}=2^{-\ell}(1+O(jn^{-1/3}))=2^{-\ell}(1+O(n^{-1/12})).
\]
Finally, from $m>\log n$ it follows that $j$ tends to infinity and thus
$\sum_{\ell=1}^j r_{\ell} = 1-o(1)$. Consequently,
 $\sum_{\ell=j}^kr_{\ell}=o(1)$.

\vskip 0.3truecm
It remains to show that we can modify our construction so that it has size $(2/3+\eps)\binom{n}{n/2}$ for some fixed $\eps>0$. In order to do so we replace some of the set systems in the canonical partition. First note that for any $\beta>0$ the sum $\sum_{i=n/4-\beta n^{1/2}}^{n/4}\binom{n/2}{i}^2$ is a positive fraction of $\sum_{i=0}^{n/4}\binom{n/2}{i}^2$. Thus we will be done if for each $i$ with $n/4-\beta n^{1/2}\le i\le n/4$ we can replace the set systems of the canonical partition that contain $i$-sets with other $i$-uniform set systems $\cH^{i}_1,\cH^{i}_2,...,\cH^{i}_{s_i}$ such that $\sum_{t=1}^{s_i}|\cH^{i}_t|^2$ is at least $(1/3+\eps)\binom{n/2}{i}^2$ for some positive $\eps$.

\vskip 0.2truecm

Frankl and F\"uredi considered in \cite{FF} the following pair of $i$-uniform intersecting set systems on a base set $Y$: let $Y$ be equipartitioned into $Y_1 \cup Y_2$ and define $$\cG^i_1=\left\{G \in \binom{Y}{i}: |Y_1 \cap G|>|Y_1|/2\right\},$$ $$\cG^i_2=\left\{G \in \binom{Y}{i}\setminus \cG_1: |Y_2 \cap G|>|Y_2|/2\right\}.$$ They observed that if $|Y|=2i+o(i^{1/2})$, then $|\cG^i_1 \cup \cG^i_2| =(1-o(1))\binom{|Y|}{i}$ and that for any $\alpha>0$ there exists $\beta>0$ such that if $|Y| \le 2i+\beta i^{1/2}$, then $|\cG^i_1 \cup \cG^i_2| \ge (1-\alpha)\binom{|Y|}{i}$.

Let us fix $0<\alpha<1/6 $ and consider $\beta$ as above. We define a modified version of the canonical partition for a given set $Y$. We replace the set systems $\cY_{i,j}$ for all
$\frac{|Y|}{2}-\frac{\beta}{2\sqrt{2}}|Y|^{1/2}\le i \le  \frac{|Y|}{2}$ and $j=1,...|Y|-2i+1$ with $\cG^i_1$ and $\cG^i_2$. As $|\cG^{i}_1|+|\cG^{i}_2|\ge (1-\alpha)\binom{n/2}{i}$, the ratio of $|\cG^i_1|^2+|\cG^i_2|^2$ and $\binom{n}{i}^2$ is at least $2(\frac{1-\alpha}{2})^2=1/2-\alpha +\alpha^2/2$ which is strictly larger than $1/3$ by choice of $\alpha$.
\end{proof}

\bigskip\noindent Katona's proof that a 2-part Sperner set
system can contain at most $\binom{n}{\lceil n/2\rceil}$ sets used
a theorem of Erd\H os \cite{Erd} on the number of sets contained
in the union of $k$ Sperner set systems. Our proofs of \tref{2I}
and Remark \ref{rem7} used \tref{kleit}, Kleitman's result on the
size of the union of $k$ intersecting families. It seems natural
to ask how large can the union of $k$ intersecting Sperner set
systems be as the problem seems to be interesting on its own right
and it might help establishing bounds on 2S2I-systems.
Unfortunately, we were only able to determine the exact result in the
very special case when $k=2$ and $n$ is odd. The result follows
easily from the following theorem of Greene, Katona and Kleitman.

\begin{theorem} [Greene, Katona, Kleitman \cite{GKK}]
\label{thm:GKK}
If $\cF\subseteq 2^{[n]}$ is an intersecting and Sperner set system, then the following inequality holds
$$
\sum_{F \in \cF, \hskip 0.2truecm |F| \le
n/2}\frac{1}{\binom{n}{|F|-1}}+\sum_{F \in \cF, \hskip 0.2truecm
|F| > n/2}\frac{1}{\binom{n}{|F|}} \le 1.
$$
\end{theorem}

\begin{corollary}
Let $\cF,\cG \subseteq 2^{[n]}$ be intersecting Sperner set systems and $n=2l+1$ an odd integer. Then we have
$|\cF \cup \cG| \le \binom{n}{l+1}+\binom{n}{l+2}$ and the inequality is sharp as shown by $\cF=\binom{[n]}{l+1}, \cG=\binom{[n]}{l+2}$.
\end{corollary}

\begin{proof} We may assume that $\cF$ and $\cG$ are disjoint. Let us add the inequality of \tref{GKK}
for both systems $\cF$ and $\cG$. The bigger the number of the
summands, the greater the cardinality of $\cF$, therefore we need
to keep the summands as small as possible to obtain the greatest
number of summands. The set size for which the summand is the
smallest is $l+1$ and the second smallest summand is for set sizes
$l$ and $l+2$. As by the disjointness of the systems the number of
smallest summands is at most $\binom{n}{l+1}$, the result follows.
\end{proof}

\section{1-part intersecting, 1-part Sperner systems}
\label{sec:1I1Ssec} In this section we study 1-part Sperner 1-part
intersecting set systems and prove \tref{1I1S}. In order to prove
the result we need a further definition. We say that the set
systems $\cF$ and $\cG$ are \textit{intersecting, cross-Sperner}
if both $\cF$ and $\cG$ are intersecting and there is no $F \in
\cF, G \in \cG$ with $F \subset G$ or $G \subset F$. We will prove
the following theorem which can be of independent interest.

\begin{theorem}
\label{thm:intcross}Let $\cF,\cG \subset 2^{[n]}$ be a pair of
cross-Sperner, intersecting set systems. Then we have
$$
|\cF|+|\cG| \le 2^{n-1}
$$
and this bound is best possible.
\end{theorem}
One of our main tools will be the following special case of the
Four Functions Theorem of Ahlswede and Daykin \cite{AD}. Let us
write $\cA \wedge \cB=\{A \cap B: A \in \cA, B \in \cB\}$ and $\cA
\vee \cB=\{A \cup B: A \in \cA, B \in \cB\}$.

\begin{theorem}[Ahlswede-Daykin,
\cite{AD}]
\label{thm:FFT}
For any pair $\cA,\cB$ of set systems  we have
$$|\cA||\cB| \le |\cA \wedge \cB||\cA \vee \cB|.$$
\end{theorem}
The other result we will use in our argument is due to Marica and
Sch\"onheim \cite{MS} and involves the difference set system
$\Delta(\cF)=\{F \setminus F': F,F'\in \cF\}$.
\begin{theorem}[Marica -- Sch\"onheim \cite{MS}]
\label{thm:diff}
For any set system $\cF$ we have $|\Delta(\cF)| \ge |\cF|$.
\end{theorem}

\begin{corollary}\label{cor:inters}
Let $\cD$ be a downward closed set system and let $\cF$ be an intersecting
subsystem of $\cD$. Then the inequality $2|\cF| \le |\cD|$ holds.
\end{corollary}

\begin{proof} As $\cD$ is downward closed and $\cF \subset \cD$, it follows that $\Delta(\cF) \subset \cD$. Furthermore, as
$\cF$ is intersecting, we have $\cF\cap \Delta(\cF)=\emptyset$ and thus we are done by \tref{diff}.
\end{proof}

\begin{proof}[Proof of \tref{intcross}.] Let us begin with defining the following four set
systems
$$
\cU=\{U \subseteq [n]: \exists H \in \cF \cup \cG \text{ such that
}\ H \subseteq U\}, \hskip 0.5truecm \cU'=\cU\setminus(\cF \cup
\cG),
$$
$$
\cD=\{D \subseteq [n]: \exists H \in \cF \cup \cG \text{ such that
}\ D \subseteq H\}, \hskip 0.5truecm \cD'=\cD\setminus (\cF \cup
\cG).
$$
Clearly, $\cD''=\{D': \exists F \in \cF \text{ such that } D'
\subset F\}$ is downward closed (and, by definition, $\cF \subset
\cD''$), hence by Corollary ~\ref{cor:inters} we have $2 |\cF| \le
|\cD''|.$ Moreover by the cross-Sperner property, we have
$(\cD''\setminus \cF) \cap \cG =\emptyset$, and therefore we have
$\cD''\setminus \cF \subset \cD'$. Consequently $|\cF| \le |\cD'|$
and, by symmetry, $|\cG| \le |\cD'|$ also holds.

Note that $\cF \wedge \cG \subset \cD'$. Indeed, $F \cap G \in
\cD$ by definition and $F \cap G \in \cF$ (or $F \cap G \in \cG$)
would contradict the cross-Sperner property. Similarly, we obtain
that $\cF \vee \cG \subset \cU'$ and it is easy to see that the
cross-Sperner property implies that $\cU' \cap \cD'=\emptyset$ and
thus the four systems $\cF,\cG, \cU',\cD'$ are pairwise disjoint.

Now suppose as a contradiction that $|\cF|+|\cG|>2^{n-1}$ and
thus $|\cU'|+|\cD'|<2^{n-1}$. By $|\cF|,|\cG|\le |\cD'|$ we obtain
that $|\cU'|<|\cF|,|\cG|$ and thus using \tref{FFT} we have
$$
|\cU'||\cD'|<|\cF||\cG|\le |\cF \wedge \cG||\cF \vee \cG| \le |\cU'||\cD'|,
$$
a contradiction.

Finally, let us mention some pairs of set systems for which the
sum of their sizes equals $2^{n-1}$. Any maximum intersecting
system $\cF$ with $\cG$ the empty set system is extremal, just as
the pair $\cF_1=\{F \subset [n]: 1 \in F, 2 \notin F\}$,
$\cG_1=\{G \subset [n]: 1 \notin G, 2 \in G\}$. Furthermore, for
any $k \ge n/2$ the pair $\cF_k=\{F \subset [n]: 1 \in F, |F| \le
k\}$, $\cG_k=\{G \subset [n]: 1 \notin G, |G| \ge k\}$ has the
required property, too.
\end{proof}

\begin{proof}[Proof of \tref{1I1S}.] First let us consider any
pair of maximal intersecting systems $\cA \subseteq 2^{X_1}$, $\cB
\subseteq 2^{X_2}$.  Clearly, the set system $\cF=\cA \times \cB$
is a 1I1S-system as any pair of sets $F_1,F_2 \in \cF$ intersect
both in $X_1$ and in $X_2$. This shows that a maximum 1I1S-system
contains at least $2^{|X|-2}$ sets.

To obtain the upper bound of the theorem let $\cF$ be any
1I1S-system. For any $A \subseteq X_1$ let $\bar{A}$ denote
$X_1\setminus A$. By definition, both $\cF_A$ and $\cF_{\bar{A}}$
are intersecting systems, and no element of the first can contain
any element of the second (and vice versa). In other words they
form a pair of intersecting, cross-Sperner systems. Due to
\tref{intcross} we have $|\cF_A| + |\cF_{\bar{A}}|\le
2^{|X_2|-1}.$  The number of pairs $A,\bar{A}$ is
$2^{|X_1|-1}$ therefore we have $|\cF| \le 2^{|X|-2}.$
\end{proof}

\end{document}